\documentclass[12pt]{amsart}

\usepackage{enumerate}

\usepackage[active]{srcltx}
\usepackage{nicefrac}

\usepackage{amsthm,amssymb,setspace}
\usepackage[pagebackref,colorlinks,linkcolor=red,citecolor=blue,urlcolor=blue,hypertexnames=true]{hyperref}
\usepackage{amsrefs}
\usepackage[matrix, arrow]{xy}

\renewcommand{\deg}{{\rm deg}}

\newcommand{\A}{\mathcal A}

\newcommand{\id}{{\rm id}}
\newcommand{\N}{\mathbb N}
\newcommand{\R}{\mathbb R}

\newcommand{\C}{\mathbb C}

\newcommand{\si}{\sigma}

\theoremstyle{plain}
\newtheorem*{theorem*}{Theorem}
\newtheorem{theorem}{Theorem}[section]
\newtheorem{corollary}[theorem]{Corollary}
\newtheorem{lemma}[theorem]{Lemma}
\newtheorem{proposition}[theorem]{Proposition}
\newtheorem{conjecture}{Conjecture}

\theoremstyle{definition}
\newtheorem*{definition*}{Definition}

\theoremstyle{remark}

\newtheorem{remark}[theorem]{Remark}

\begin{document}

\onehalfspace

\title{Hyperbolic Polynomials and Generalized Clifford Algebras}

\author{Tim Netzer}
\address{Tim Netzer, Universit\"at Leipzig, Germany}
\email{netzer@math.uni-leipzig.de}

\author{Andreas Thom}
\address{Andreas Thom, Universit\"at Leipzig, Germany}
\email{thom@math.uni-leipzig.de}

\begin{abstract} We consider the problem of realizing hyperbolicity cones as spectrahedra, i.e. as linear slices of cones of positive semidefinite matrices.  The generalized Lax conjecture states that this is always possible. We use generalized Clifford algebras for a new approach to the problem. Our main result is that if $-1$ is not a sum of hermitian squares in the Clifford algebra of a hyperbolic polynomial, then  its hyperbolicity cone is spectrahedral. Our result also has computational applications, since this sufficient condition can be checked with a single semidefinite program. \end{abstract}

\maketitle

\tableofcontents

\section{Introduction} A homogeneous polynomial $h\in\R[x]=\R[x_1,\ldots,x_n]$ is called {\it hyperbolic in direction $e\in\R^n$}, if $h(e)\neq 0$ and the univariate polynomials $$h_{a,e}(t):=h(a-t\cdot e)$$ have only real roots, for all $a\in\R^n$. The {\it hyperbolicity cone of $h$ in direction $e$} is $$\Lambda_{e}(h)=\{a\in\R^n\mid h_{a,e}(t)\ \mbox{Êhas only nonnegative roots} \}.$$ It can be shown (see \cite{gar} and also \cite{ren}) that $\Lambda_e(h)$ is a convex cone with  interior $$\mathring{\Lambda}_e(h)=\{a\in\R^n\mid h_{a,e}(t)\ \mbox{Êhas only positive roots}\},$$ which also coincides with the connected component of $\{a\in\R^n\mid h(a)\neq 0\}$ containing $e.$ Furthermore, $h$ is hyperbolic in direction $e'$ for all $e'\in\mathring{\Lambda}_e(h)$, and the hyperbolicity cones coincide: $$\Lambda_{e'}(h)=\Lambda_e(h)\quad\mbox{Êand }\quad\mathring\Lambda_{e'}(h)=\mathring\Lambda_e(h).$$ 
A polynomial, hyperbolic in direction $e$, is said to have a {\it definite determinantal representation} if there are hermitian matrices $M_1,\ldots,M_n\in{\rm Her}_d(\C)$ with $$h=\det\left(x_1M_1+\cdots +x_nM_n\right)$$ and $e\bullet M:=e_1M_1+\cdots +e_nM_n$ strictly definite. Note that by homogeneity, the size $d$ of the matrices $M_i$ coincides with the degree of $h$.
Further note that from such a representation the hyperbolicity of $h$ is obvious. Indeed if we assume $e\bullet M=I$ (which we can if we assume $h(e)=1$ and after conjugation with a regular matrix) then $h_{a,e}(t)$ is the characteristic polynomial of the hermitian matrix $a\bullet M,$ and has thus only real roots. Also note that $$\Lambda_e(h)=\{ a\in\R^n\mid a\bullet  M\succeq 0\}\quad \mbox{and}\quad  \mathring{\Lambda}_e(h)=\{a\in\R^n\mid a\bullet  M\succ 0\}$$ in this case. Thus the hyperbolicity  cone is {\it spectrahedral}, i.e. defined by a linear matrix inequality, i.e. a linear section of the cone of positive semidefinite matrices.

The interest in hyperbolic polynomials arose in  the area of partial differential equations (see for example \cite{gar,lax}). In recent time, it has attracted attention in optimization, especially semidefinite optimization (see for example \cite{gue,wo,ren,hevi,past}), and also combinatorics (see for example \cite{cosw,br1}). In particular, the question whether a hyperbolic polynomial has a definite determinantal representation, or more generally whether the hyperbolicity cone is spectrahedral, has recently been discussed extensively. We state some of the known results. 

In the case $n=3,$ every hyperbolic polynomial admits a definite determinantal representation, even with real symmetric matrices. This is the main result from \cite{hevi}, which solves a conjecture of Peter Lax, going back to the 50's (\cite{lax}, see also \cite{lepara}). The same result cannot be true in higher dimensions, as is easily seen by a count of parameters. So several possible generalizations to higher dimensions have been discussed. One guess was that always some power $h^r$ of a hyperbolic polynomial admits a definite determinantal representation. This is true for quadratic polynomials \cite{neth}, but turns out to be false in general \cite{br1}. The following version, which has now become known as the {\it Generalized Lax Conjecture}, is still open: \begin{conjecture}[Generalized Lax Conjecture] \label{genlax}Every hyperbolicity cone is spectrahedral, i.e. of the form $$\Lambda_e(h)=\{a\in\R^n\mid a\bullet M\succeq 0\}$$ for some hermitian matrices $M_1,\ldots,M_n$ (i.e. a linear section of the cone of positive semidefinite matrices). Alternatively, for every hyperbolic polynomial $h$ there is another hyperbolic polynomial $f$ such that $\Lambda_e(fh)=\Lambda_e(h)$ and $fh$ admits a definite determinantal representation. \end{conjecture}

Since $f$ can in general not be chosen to be some power of $h$, several other options have been considered. In fact, one could take $f$ to be some power of $e\bullet x=e_1x_1+\cdots +e_nx_n.$ At least if the intersection of $\Lambda_e(h)$ with the affine plane $\{e\bullet x=1\}$ is compact, this would not change the hyperbolicity cone. However, for almost no hyperbolic polynomial,  such a multiple will admit a definite determinantal representation. This is a straightforward translation of a result in  \cite{neth}. 

If $h$ is hyperbolic in direction $e$, then so is its directional derivative $\partial_e(h)$, and the hyperbolicity cone of $\partial_e(h)$ contains $\Lambda_e(h).$ This follows immediately from Rolle's theorem. So one could multiply $h$ with some of its (maybe higher) directional derivatives to obtain a definite determinantal representation, without changing the hyperbolicity cone. It was shown in \cite{br2} that this works for the elementary symmetric polynomials (which are hyperbolic in direction $e=(1,\ldots,1)$). In particular, their cones a spectrahedral, and thus all higher directional derivative cones of polyhedral cones are spectrahedral as well.  This was shown before for the first such derivative cone in \cite{san}. For more on these results, see \cite{vin} for an up-to-date overview as of 2012.

In this paper we develop a new method to produce determinantal representations, and show that hyperbolicity cones are spectrahedral. It involves the generalized Clifford algebra of the polynomial $h$, also used in \cite{neth}. We show that if $-1$ is not a sum of hermitian squares in this algebra, then the hyperbolicity cone of $h$ is spectrahedral (Theorem \ref{main}). An extra factor $f$ in a determinantal representation of $h$ emerges naturally in our construction. The method is also interesting from a computational point of view. In fact a {\it single} semidefinite program can be constructed canonically from $h$, and if this program is feasible, then the hyperbolicity cone is spectrahedral. If the program is unfeasible, then no power of $h$ admits a definite determinantal representation (Theorem \ref{comp}). 
Also the multivariate Hermite matrix of $h$, that has been considered for example in \cite{neplth}, fits well into the framework of generalized Clifford algebras. We show in Section \ref{hertr} how sums of squares decompositions of this matrix are related to a trace on the algebra. 

We finally suggest an abstract version of the generalized Lax conjecture. We conjecture that  the hyperbolicity cone of $h$ always arises as a canonical linear section with the sums of squares cone in the generalized Clifford algebra of $h$ (Conjecture \ref{gengenlax}). Our results imply that a positive answer to this conjecture implies a positive answer to the generalized Lax conjecture.

\section{Generalized Clifford Algebras}\label{eins}
Let $h\in\R[x]$ be hyperbolic in direction $e$.  The generalized Clifford algebra of $h$, used in \cite{neth} before, is a universal model for the algebra generated by the matrices representing $h$, in case they exist. The key observation is that $h_{a,e}(t)$ is the characteristic polynomial of $a\bullet M,$ if $h=\det(x_1M_1+\cdots +x_nM_n)$ is a positive determinantal representation of $h$. By the Cayley-Hamilton Theorem we thus have $h_{a,e}(a\bullet M)=0.$ 

Now let $\C\langle z\rangle=\C\langle z_1,\ldots,z_n\rangle $ be the free noncommutative  algebra equipped with the involution defined by $z_i^*=z_i.$ Let $J_e(h)$ be the two-sided ideal generated by the elements $$h_{a,e}(a\bullet   z)=h_{a,e}(a_1z_1+\cdots +a_nz_n)\in \C\langle  z\rangle$$ for  $a\in\R^n,$ and the additional element $1-e\bullet z$. Note that since $h_{e,e}(t)=(1-t)^dh(e)$, this extra element is just the reduction of the element $h_{e,e}(e\bullet z),$ that we have anyway.  Now  $J_e(h)$ is a $*$-ideal, which is in fact finitely generated. Indeed if $$h_{a,e}(a\bullet  z)=\sum_{\alpha} q_\alpha( z)a^\alpha$$ then $J_e(h)$ is generated by the finitely many elements $q_\alpha$ (and $1-e\bullet z$).
The quotient $$\A_e(h):=\C\langle z\rangle / J_e(h)$$ is a unital $*$-algebra, called the {\it generalized Clifford algebra associated with $h$ (and e)}.
We denote the residue class of $z_i$ in $\A_e(h)$  by $\si_i$, and the tuple $(\si_1,\ldots,\si_n)$ by $\si$. We have $h_{a,e}(a\bullet\si)=0.$ Let $$\Sigma^2 \A_e(h)=\left\{ \sum_i a_i^*a_i\mid a\in \A_e(h)\right\}$$ denote the convex cone of sums of hermitian squares in $\A_e(h).$

\begin{remark}
A result from \cite{bhs} implies that for every homogeneous $h$ there is a representation $$h^r=\det(x_1M_1+\cdots +x_nM_n)$$ for some $r\geq 1$,  where the $M_i$ are not necessarily hermitian matrices. If $h(e)\neq 0$ we can then even assume $e\bullet  M=I,$ after scaling $h$ suitably. Then $h^r_{a,e}(t)$ is the characteristic polynomial of $a\bullet  M$ and thus $h^r_{a,e}(a\bullet  M)=0$ holds. So at least $\A_e(h^r)$ admits a unital algebra homomorphism to a matrix algebra and is thus not trivial.\end{remark}

\begin{lemma}\label{help} Let $h$ be hyperbolic in direction $e.$

(i) If $h(a)\neq 0$, then $a\bullet \sigma$ is invertible in $\A_e(h).$

(ii) If $a\in\mathring{\Lambda}_e(h)$, then $a\bullet  \si\in\Sigma^2 \A_e(h).$ 

(iii) $1$ is an algebraic interior point of $\Sigma^2 \A_e(h).$ 
\end{lemma}
\begin{proof}
For any $a\in\R^n$ we have the relation $h_{a,e}(a\bullet  \si)=0$ in $\A_e(h).$ Since $h_{a,e}(t)$ has constant term $h(a)$, the first statement is clear. If $a\in\mathring{\Lambda}_e(h)$, then $h_{a,e}(t)$ has only positive roots, so there is an identity $$t=f(t)^2+q(t)h_{a,e}(t)$$ in the univariate polynomial ring $\R[t].$  Plugging in $a\bullet  \si$ for $t$ proves the second claim. 
For the third use that $e\pm \epsilon v\in \mathring{\Lambda}_e(h)$ for all unit vectors $v\in\R^n$ and $\epsilon>0$ sufficiently small. Thus $1\pm \epsilon (v\bullet \si)=e\bullet  \si \pm \epsilon\cdot (v\bullet  \si)=(e \pm\epsilon v) \bullet \si \in \Sigma^2 \A_e(h)$ by the second statement. 
This is well known to imply that  $1$ is an interior point, see for example \cite{cim}. \end{proof}

The following was proven in \cite{neth} for real zero polynomials. We sketch the proof, slightly adapted to the hyperbolic setup.

\begin{theorem}\label{power}
Let $h$ be irreducible and hyperbolic in direction $e$. The following are equivalent:\begin{itemize} \item[(i)] Some power $h^r$ admits a  definite determinantal representation. \item[(ii)] $\A_e(h)$ admits a finite-dimensional $*$-representation, i.e. a unital  $*$-algebra homomorphism $\pi\colon\A_e(h)\rightarrow {\rm M}_k(\C).$\end{itemize}
\end{theorem}
\begin{proof}
``(i)$\Rightarrow$ (ii)'': Assume $h^r=\det(x_1M_1+
\cdots +x_nM_n)$ with all $M_i$ hermitian and  $e\bullet  M$ strictly definite. After multiplying $h$ with a nonzero real number (which does not change the algebra) and a base change we can assume $e\bullet  M=I$, which then implies that $h^r_{a,e}(t)$ is the characteristic polynomial of $a\bullet  M$, for all $a\in\R^n.$ This implies $h_{a,e}(a\bullet  M)=0$, and we thus get a finite-dimensional $*$-representation of $\A_e(h)$ by plugging in $M_i$ for $\si_i.$ ``(ii)$\Rightarrow$(i)'': Assume $\pi\colon \A_e(h)\rightarrow {\rm M}_k(\C)$ is a $*$-algebra homomorphism, and set $M_i:=\pi(\si_i).$ Get get $h_{a,e}(a\bullet  M)=0$ for all $a\in\R^n$, and $e\bullet  M=I.$
Now set $g:=\det (x_1M_1+\cdots +x_nM_n)$. The zeros of $g_{a,e}(t)$ are the 
eigenvalues of $a\bullet  M$, which are contained in the zeros of $h_{a,e}(t).$ So $h$ vanishes on the real zero set of $g$, and this implies $h^r=g$, using the real Nullstellensatz and irreducibility of $h$.
\end{proof}

\begin{remark} (i) The main result of \cite{hevi} implies that $\A_e(h)$ always has a finite-dimensional $*$-representation in the case of $n=3$.

(ii) In \cite{neth} it was shown that $\A_e(h)$ admits a finite-dimensional $*$-representation if $h$ is of degree $2.$ So some power $h^r$ has a definite determinantal representation and $\Lambda_e(h)$ is spectrahedral in this case.
\end{remark}

From Theorem \ref{power} we see that whether $\A_e(h)$ admits a finite-dimensional $*$-representation does not depend on the direction $e$ (from the interior of the hyperbolicity cone). Given a representation of $\A_e(h),$ we get a determinantal representation of $h^r$ adapted to the direction $e$. This representation can be transformed to a representation adapted to $e'$, by replacing the matrices $M_i$ by $P^*M_iP$ for some invertible matrix $P$. This then yields a $*$-representation of $\A_{e'}(h),$ using the Cayley-Hamilton Theorem. A similar transformation procedure can however be done in the algebras directly, see Theorem \ref{iso} below. For this we  need an abstract version of the Cayley-Hamilton Theorem. In the case of a matrix algebra $\A$, the element $b$ from (ii) in the following Lemma is the adjugate matrix of $t-a$, and $p$ is the characteristic polynomial of the matrix $a$.

\begin{lemma}[Cayley-Hamilton]\label{ch}
Let $\A$ be a unital complex algebra and $a\in \A$, $p\in\C[t].$ The following are equivalent:\begin{itemize}
\item[(i)] $p(a)=0$ holds in $\A$. \item[(ii)] There is some $b\in \A[t]$ with $(t-a)b=p.$
\end{itemize} 
\end{lemma}
\begin{proof}
(i)$\Rightarrow$(ii): Consider $p(t+s)=\sum_{i=0}^d q_i(s)t^i \in\C[s,t],$ where $q_0(s)=p(s),$  and set $$f:=\sum_{i=1}^d q_i(s)(t-s)^{i-1}.$$ Then $$(t-s)f+p(s)=p(t)$$ holds in $\C[s,t],$ and we obtain (ii) by plugging in $a$ for $s$. (ii) $\Rightarrow$(i) is precisely one of the standard algebraic proofs of the Cayley-Hamilton Theorem. Write $p=\sum_{i=0}^dp_it^i$ and $b=\sum_{i=0}^{d-1} b_i t^i$ with $p_i\in \C, b_i\in\A$. Then $$(t-a)b= \sum_{i=1}^d b_{i-1} t^i - \sum_{i=0}^{d-1}ab_i t^i.$$ Comparing coefficients with $p$ shows $$b_{d-1}=p_d, \quad b_{i-1}-ab_i=p_i \mbox{ for } i=1,\ldots,d-1, \quad  -ab_0=p_0.$$ We multiply the equation for each $p_i$ from the left with $a^i$ and sum up. The left-hand side is a telescope sum that cancels completely, and the right-hand side equals $p(a).$
\end{proof}

We see in the proof of (i) $\Rightarrow$(ii) that $b$ can be chosen in a very specific way. Indeed $b=\sum_i b_i(a)t^i,$ for some $b_i\in\R[s]$ whose coefficients are polynomial expressions in the coefficients of $p$. We will use this in the following proof.

\begin{theorem}\label{iso} Let $h$ be hyperbolic in direction $e$, and $e'\in \mathring{\Lambda}_e(h).$ Then $\A_{e'}(h)\cong \A_e(h)$ as unital $*$-algebras.
\end{theorem}
\begin{proof}Since $e'\in\mathring\Lambda_e(h)$ we have an equation $e'\bullet \si=w^*w=ww^*=w^2$ for some invertible $w\in\A_e(h).$ This follows from Lemma \ref{help} (i) and the proof of (ii). In fact $w$ is a real polynomial expression in $e'\bullet \si.$ Let $v=w^{-1}$ and consider the unital $*$-algebra homomorphism \begin{align*} \psi\colon \C\langle z\rangle &\rightarrow \A_e(h) \\ z_i &\mapsto v^*\si_i v \end{align*} To see that $\psi$ factors through $\A_{e'}(h)$ we have to prove $h_{a,e'}(a\bullet v^*\si v)=0$ in $\A_e(h),$ for all $a\in\R^n$.
In the algebra $\A_e(h)[t]$ we have the following equations, for all $a\in\R^n$: $$ (t-a\bullet \si) \sum_i b_i(a,a\bullet\si)t^i=h_{a,e}(t),$$ with certain  $b_i\in\C[x,t].$ This follows from Lemma \ref{ch} and the fact that $h_{a,e}(a\bullet \si)=0$.
Now consider the equation $$(t-x\bullet\si)\sum_i b_i(x,x\bullet\si)t^i=h(x-te)$$ in $\A_e(h)[t,x],$ which holds since it holds for every evaluation $x\mapsto a$ for  $a\in\R^n.$ We plug in $a-t(e'-e)$ for $x$, which we can, since it commutes with everything.
We obtain $$(tw^*w -a\bullet\si)q=h_{a,e'}(t)$$ for some $q\in\A_e(h)[t],$ using $e\bullet\si=1.$ We multiply with $v^*$ from the left and $w^*$ from the right. The right-hand side of the equation does not change, since $h_{a,e'}(t)$ commutes with everything and $v^*w^*=1.$ The left-hand side becomes $(t-a\bullet v^*\si v)wqw^*.$ We can now apply Lemma \ref{ch} and obtain $h_{a,e'}(a\bullet v^*\si v)=0,$ the desired result.

To see that $\psi\colon \A_{e'}(h)\rightarrow \A_e(h)$ is an isomorphism, we first observe $\psi(e\bullet \si)=v^*v=(e'\bullet\si)^{-1}.$  Note that we also have a homomorphism $\varphi\colon \A_e(h)\rightarrow \A_{e'}(h)$ with $\varphi(e'\bullet\si)=(e\bullet\si)^{-1},$ by the same argument. Since these inverse elements are polynomial expressions in the elements themselves (see Lemma \ref{help} (i) again), we have homomorphisms $$\psi\colon \C[e\bullet\si]\rightarrow\C[e'\bullet\si] \mbox{ and }  \varphi\colon\C[e'\bullet\si]\rightarrow\C[e\bullet\si]$$ of the corresponding subalgebras, which are inverse to each other.
For the element $w\in\C[e'\bullet\si]$ that we used to define $\psi$, we obtain $$\varphi(w)^*\varphi(w)=\varphi(w^*w)=\varphi(e'\bullet\si)=(e\bullet\si)^{-1}.$$ So $\varphi(w)^{-1}$ is a square root of $e\bullet\si$ in $\A_{e'}(h)$, and by repeating the first part of the proof, we find (a possibly different) homomorphism $\tilde\varphi\colon\A_e(h)\rightarrow \A_{e'}(h)$ that maps $\si$ to $\varphi(w)^* \si\varphi(w).$ Since $\tilde\varphi$ is also inverse to $\psi$ on $\C[e\bullet\si],$ we have $\tilde\varphi=\varphi$ on $\C[e'\bullet\si]$. This implies that $\tilde\varphi$ is indeed a global inverse to $\psi,$ as one easily checks.\end{proof}

\section{The  Main Result}
We have seen how finite-dimensional $*$-representations of $\A_e(h)$ correspond to determinantal representations of powers of $h$. We can however also consider $*$-representations which are not finite-dimensional. In the below main result we use such representations to produce spectrahedral representations of the hyperbolicity cone. To get such infinite-dimensional representations, we only need that $-1$ is not a sum of squares in the Clifford algebra.

\begin{theorem}\label{main}
Assume $h$ is irreducible and hyperbolic in direction $e$. If  $-1\notin \Sigma^2\A_e(h),$ then $\Lambda_e(h)$ is spectrahedral.
\end{theorem}
\begin{proof}
If $-1\notin\Sigma^2\A_e(h)$, then there is a linear functional $\varphi\colon\A_e(h)\rightarrow \C$ with $\varphi(1)=1,\varphi(a^*)=\overline{\varphi(a)}$ and $\varphi(a^*a)\geq 0$ for all $a\in\A_e(h).$ This is the  Hahn-Banach separation  theorem, using that $1$ is an algebraic interior point of $\Sigma^2\A_e(h)$, by Lemma \ref{help}. We can now do the standard GNS construction with $\varphi$ and obtain an inner product space $H$ and a unital $*$-algebra homomorphism $$\pi\colon \A_e(h)\rightarrow  \mathcal{B}(H)$$ of $\A_e(h)$ into the bounded linear operators on $H$. We set $T_i:=\pi(\si_i),$ a self-adjoint operator, and note that $$h_{a,e}(a\bullet  T)=0$$ as well as $e\bullet T={\id}_H$ holds in $\mathcal{B}(H).$  Now fix some $0\neq v\in H$ and let $$H'=\{ p(T)v\mid p\in\C\langle z\rangle, \deg(p)\leq d-1\}.$$ $H'$ is a finite-dimensional subspace of $H,$ we consider the orthogonal projection ${\rm pr}\colon H\rightarrow H'$ and the self-adjoint operators $$M_i:={\rm pr}\circ T_i\colon H'\rightarrow H'.$$ We have $e\bullet  M={\rm pr}\circ (e\bullet T)= {\id_{H'}}$. Since $h_{a,e}(t)$ is of degree $d$, one checks that $$h_{a,e}(a\bullet  M)v={\rm pr}\left(h_{a,e}(a\bullet  T)v\right)=0$$ holds. So each $a\bullet  M$ has at least one eigenvalue which is among the zeros of $h_{a,e}(t).$ Since for $g:=\det(x_1M_1+\cdots +x_nM_n)$  the polynomial $g_{a,e}(t)$ is the characteristic polynomial of $a\bullet M$, each $g_{a,e}(t)$ has at least one common zero with $h_{a,e}(t).$ By irreducibility of $h$ and the real Nullstellensatz, this implies $g=fh.$ We finally prove $\Lambda_e(h)=\Lambda_e(g).$ The inclusion $\supseteq$ is obvious. For $\subseteq$ take $a\in \mathring{\Lambda}_e(h).$ By Lemma \ref{help}, $a\bullet \si\in\Sigma^2\A_e(h),$ and so $a\bullet T$ is a positive semidefinite operator on $H$. Then $a\bullet M={\rm pr}\circ(a\bullet T)$ is positive semidefinite on $H'$ as well, which implies $a\in\Lambda_e(g).$ This finishes the proof.\end{proof}

\begin{remark}(i) Note that $-1\notin\Sigma^2\A_e(h)$ is equivalent to having a $*$-representation of $\A_e(h)$ on some inner product space, possibly infinite-dimensional. This uses that $-1$ is an interior point of the sums of squares cone.

 (ii) Note that $-1\notin\Sigma^2\A_e(h)$ is true in the case that some power of $h$ admits a definite determinantal representation. In fact  there is a finite-dimensional $*$-representation of $\A_e(h)$ then, and this shows $-1\notin\Sigma^2 \A_e(h).$ 

(iii) From the proof we get an upper bound for the size of the matrices $M_i$, and thus for the degree of the factor $f$ that appears in the determinantal representation of $h$. The size of the $M_i$ is the dimension of $H'$, which is at most $$\dim_\C \C\langle z\rangle_{d-1}=\frac{n^d-1}{n-1}.$$

(iv) Combining (ii) and (iii) we see that if any large power $h^r$ admits a determinantal representation, then also some multiple $fh$, where we have control over the degree of $f$. \end{remark}
The following is an abstract version of the generalized Lax conjecture. It states that each hyperbolicity cone is a canonical linear section of the (closed) sums of squares cone in the generalized Clifford algebra. The closure is taken with respect to the finest locally convex topology, and equals the double dual cone.
\begin{conjecture}\label{gengenlax}
Let $h$ be hyperbolic in direction $e$. Consider the linear map $$\iota\colon \R^n\rightarrow \A_e(h),\quad  a\mapsto a\bullet \si.$$ Then $$\Lambda_e(h)=\iota^{-1}\left(\overline{\Sigma^2\A_e(h)}\right).$$
\end{conjecture}

\noindent
From Theorem \ref{main} we easily get the following implication:

\begin{corollary}
If Conjecture \ref{gengenlax} Êis true for a polynomial $h\neq 1$, then also Conjecture \ref{genlax} is true for $h$.
\end{corollary}
\begin{proof}
If Conjecture \ref{gengenlax} is true  for $h$, then clearly $-1\notin\Sigma^2\A_e(h).$ Indeed if $-1$ was a sum of hermitian squares, then every hermitian element would be a sum of hermitian squares, contradicting the fact that $\Lambda_e(h)\neq\R^n.$ So Theorem \ref{main} shows that $\Lambda_e(h)$ is spectrahedral. 
\end{proof}

\begin{remark}

(i) The inclusion $\subseteq$ in Conjecture \ref{gengenlax} is always true, as shown in Lemma \ref{help}.

(ii) In case that some power $h^r$ admits a definite determinantal representation, Conjecture \ref{gengenlax} is true for $h$. In fact use Theorem \ref{power} to construct a finite-dimensional $*$-representation $\pi$ of $\A_e(h)$ and consider $$\R^n\overset{\iota}{\rightarrow} \A_e(h)\overset{\pi}{\rightarrow} {\rm M}_k(\C).$$ If $a\in\iota^{-1}\left(\overline{\Sigma^2\A_e(h)}\right),$ then $\iota(a)=a\bullet \si\in\overline{\Sigma^2\A_e(h)}$ and thus $\pi(\iota(a))=a\bullet\pi(\si)\succeq 0$. Since $h^r=\det(x\bullet \pi(\si))$ we have $\Lambda_e(h)=\{a\in\R^n\mid a\bullet \pi(\si)\succeq 0\}.$ This proves the other inclusion.
\end{remark}

\section{The Hermite Matrix and the Trace}\label{hertr}

In our main theorem from the last section, we used a positive functional on $\A_e(h)$ to construct a determinantal representation of some multiple of $h$.  In case that $\A_e(h)$ has a finite-dimensional $*$-representation, we have a distinguished such functional, namely the trace. We can try to reconstruct this trace on $\A_e(h)$ in general. It turns out that there is a close  connection to the Hermite matrix of $h$ (which was considered in detail in \cite{neplth}). 

For a monic univariate polynomial $p\in\R[t]$ of degree $d$ with (complex) zeros $\lambda_1,\ldots,\lambda_d,$ the $k$-th {\it Newton sum} is $N_k(p)=\sum_{j=1}^d\lambda_j^k.$ The Newton sums are polynomial expressions in the coefficients of $p$. The Hermite matrix of $p$  is $$\mathcal{H}(p)=\left(N_{i+j}(p)\right)_{i,j=0,\ldots,d-1}.$$ The entries of the  Hermite matrix are polynomials in the coefficients of $p$, and $\mathcal{H}(p)$  is positive semidefinite if and only if all $\lambda_j$ are real.

Now let $h$ be homogeneous with $h(e)\neq 0.$ We define $\mathcal{H}_e(h)$ to be the Hermite matrix of $h(x-te)$ as a univariate polynomial in $t.$ So the entries of $\mathcal{H}_e(h)$ are polynomials in the variables $x,$ and $h$ is hyperbolic in direction $e$ if and only of $\mathcal{H}_e(h)(a)\succeq 0$ for all $a\in\R^n.$ It was shown in \cite{neplth} that $\mathcal{H}_e(h)$ is even a sum of hermitian squares of polynomial matrices, if some power  of $h$ admits a definite determinantal representation. We will generalize this in the following.

First assume that $h=\det(x_1M_1+\cdots +x_nM_n)$ is a determinantal representation with $e\bullet M=I$. Then for all $a\in\R^n$ $$N_k(h_{a,e}(t))={\rm tr}\left((a\bullet M)^k\right).$$ Expanding as polynomials in $a$ on both sides yields an identity  $$\sum_{\vert\alpha\vert =k} c_\alpha \cdot a^\alpha = \sum_{\vert\alpha\vert =k} {\rm tr}(H_\alpha(M)) \cdot a^\alpha,$$ where $H_\alpha\in\C\langle z\rangle$ is the sum over all words in $z$ of commutative type $\alpha$. The coefficients $c_\alpha$ on the left hand side are determined by $h$ and $e$ alone. Thus we know ${\rm tr}\left((H_\alpha(M)\right)$ without knowing $M$. This motivates the following result:

\begin{proposition} \label{trace} Assume there is a positive  trace functional $\varphi\colon\C\langle z\rangle\rightarrow \C$ with $\varphi(H_\alpha)=c_\alpha$ for all $\alpha\in\N^n,$ $\vert\alpha\vert\leq 2d$. Then $\Lambda_e(h)$ is spectrahedral.\end{proposition}

\begin{proof} We show that $\varphi(h^2_{a,e}(a\bullet z))=0$ for all $a\in\R^n$. Using the trace property $\varphi(vw)=\varphi(wv)$ and the Cauchy-Schwarz inequality we then see that $\varphi$ vanishes on the ideal $J_e(h),$ and thus  defines a positive functional on $\A_e(h).$

Write $h_{a,e}(t)=\sum_{i=0}^d q_i(a)t^i$ as a polynomial in $t$ and compute $$\varphi(h_{a,e}^2(a\bullet z))=\sum_{i,j=0}^d q_i(a)q_j(a) \varphi\left( (a\bullet z)^{i+j}\right)=\sum_{i,j=0}^d q_i(a)q_j(a)N_{i+j}(h_{a,e}(t)).$$
The expression on the right is the Hermite matrix of size $d+1$ of $h_{a,e}(t)$, multiplied from both sides with the vector $(q_0(a),\ldots,q_d(a)).$ Since this Hermite matrix is $V^tV$ , where $V$ is the Vandermonde Matrix of size $d \times (d + 1)$ of the zeros of $h_{a,e}(t)$, this proves that the expression is $0$. 
\end{proof}

On the other hand we have the following obstruction for the existence of such a $\varphi$:

\begin{proposition} \label{hermite} Assume there is a positive  linear functional $\varphi\colon\C\langle z\rangle\rightarrow \C$ with $\varphi(H_\alpha)=c_\alpha$ for all $\alpha\in\N^n,\vert\alpha\vert \leq 2(d-1)$. Then the Hermite matrix $\mathcal{H}_e(h)$ is a sum of hermitian squares of polynomial matrices.
\end{proposition}
\begin{proof}
First note that $\varphi$ is completely positive, which means that $\left( \varphi(p_i^*p_j)\right)_{i,j}$ is positive semidefinite, for any $p_1,\ldots,p_m \in \C \langle z\rangle$. This follows from linearity and positivity of $\varphi.$ Thus for any $*$-algebra $\mathcal{B},$ the linear mapping $$\id\otimes\varphi \colon \mathcal{B}\otimes_\C \C\langle z\rangle\rightarrow \mathcal{B}$$ is positive in the sense, that it maps sums of hermitian squares to sums of hermitian squares. We apply this to the case $\mathcal{B}={\rm M}_d(\C[x])$ and obtain the positive mapping $${\rm id}\otimes\varphi\colon {\rm M}_d(\C[x])\otimes_\C \C\langle z \rangle\cong {\rm M}_d\left( \C[x]\otimes_\C \C\langle z\rangle\right) \rightarrow {\rm M}_d(\C[x]).$$ Note that $\left((x\bullet z)^{i+j}\right)_{i,j=0,\ldots,d-1}$ is a sum of hermitian squares in  ${\rm M}_d\left( \C[x]\otimes_\C \C\langle z\rangle\right),$ which is mapped to $\mathcal{H}_e(h)$ under ${\rm id}\otimes \varphi.$
\end{proof}

\begin{corollary}[\cite{neplth}]
If some power $h^r$ admits a definite determinantal representation, then $\mathcal{H}_e(h)$ is a sum of hermitian squares of polynomial matrices.
\end{corollary}
\begin{proof}
As we have explained above, the trace we obtain from a determinantal representation will have the desired properties  from the last Proposition.
\end{proof}

\section{Computational Aspects}
In the proof of Theorem \ref{main} we use a positive linear functional on $\A_e(h)$ to construct a definite determinantal representation of some multiple of $h$. In fact the functional only needs to be defined on a finite-dimensional subspace for the argument to work, as we now explain.

Let $\C\langle z\rangle_m$ be the finite-dimensional space of noncommutative polynomials of degree at most $m$. Let $h$ be hyperbolic of degree $d$, and set $k=2(d-1).$ Now assume  we have a linear functional $$\varphi\colon \C\langle z\rangle_{2k}\rightarrow \C$$ with \begin{itemize}\item $\varphi(1)=1$ and $\varphi(p^*)=\overline{\varphi(p)}$ for all $p\in\C\langle z\rangle_{2k}$ \item $\varphi(p^*p)\geq 0$ for all $p\in\C\langle z\rangle_k$ \item $\varphi(p\cdot h_{a,e}(a\bullet z)\cdot q)=0$ for all $p,q\in\C\langle z\rangle$ such that $\deg(p\cdot h_{a,e}(a\bullet z)\cdot q)\leq 2k.$ \end{itemize} Note that finding such $\varphi$ amounts to solving a {\it single} semidefinite feasibility problem, canonically constructed from $h$ (and $e$).

With $\varphi$ we now perform a partial GNS construction  and equip $\C\langle z\rangle_k$ with the bilinear form $\langle a,b\rangle =\varphi(b^*a).$ After passing to the quotient with respect to $N=\{a\mid \langle a,a\rangle =0\}$ this becomes an inner product.  We now consider the hermitian linear operators $$M_i\colon \C\langle z\rangle_{d-1}/N\rightarrow \C\langle z\rangle_d /N \rightarrow \C\langle z\rangle_{d-1}/N,$$ where the first map is multiplication with (the residue class of) $z_i$ and the second is the orthogonal projection. For the well-definedness of multiplication with $z_i$ we use the Cauchy-Schwarz inequality, and need $\varphi$ defined on $\C\langle z\rangle_{2(d+1)}$. As in the proof of Theorem \ref{main} one then checks that $$h_{a,e}(a\bullet M)v=0,$$ where $v$ is the residue class of $1.$ We need here that $h_{a,e}(a\bullet z)\in N,$ which follows from the third condition on $\varphi.$ The case $a=e$, together with the Cauchy-Schwarz inequality, shows  $e\bullet M={\id}.$ Now the operators $M_i$ will give rise to a determinantal representation of some multiple of $h$. To see that the extra factor does not change the hyperbolicity cone, use an equation $t=f(t)^2+q(t)h_{a,e}(t)$ for $a\in\mathring{\Lambda}_e(h)$ as in Lemma \ref{help} again. Here $f$ can be chosen of degree $d-1$ and $q$ thus of degree $d-2.$ Using this equation one checks $\langle (a\bullet M)w,w\rangle\geq 0$ for all residue classes $w$ of elements from $\C\langle z\rangle_{d-1}.$ To obtain this, $\varphi$ needs to be defined up to degree $4(d-1),$ and we need the third property of $\varphi$ again. All in all we have:

\begin{theorem}\label{comp} Let $h$ be irreducible and hyperbolic in direction $e$. There is a canonical spectrahedron $F(h,e)$ (the set of all positive functionals on $\C\langle z\rangle_{2k}$ fulfilling the above conditions), which has the following property: \begin{itemize}\item If $F(h,e)\neq \emptyset$, then $\Lambda_e(h)$ is spectrahedral. \item If $F(h,e)=\emptyset$ then $-1\in\Sigma^2\A_e(h)$, and in particular no power of $h$ admits a definite determinantal representation. \end{itemize}  \end{theorem}

\section*{Some Open Questions}

We conclude our paper with some open questions.

\begin{itemize}
\item[1.] Is Conjecture \ref{gengenlax} true? This would imply the generalized Lax conjecture.
\item[2.] Is $-1\notin\Sigma^2 \A_e(h)$ for any hyperbolic $h$? This would also imply the generalized Lax conjecture.
\item[3.] Is $-1\notin \Sigma^2 \A_e(h)$ equivalent to Conjecture \ref{gengenlax}?
\item[4.] If $\mathcal H_e(h)$ is a sum of squares, does a functional $\varphi$ as in Proposition \ref{trace} exist?
\item[5.]The Br\"and\'{e}n-Vamos polynomial $h$ from \cite{br2} is an example of a polynomial of which no power admits a determinantal representation. Is $-1\in\Sigma^2 \A_e(h)$ in this case? Is $F(h,e)=\emptyset$?
\end{itemize}

\end{document}